\documentclass[12pt,letterpaper]{article}

\usepackage{amsmath,amssymb,amsthm,mathtools}
\usepackage[margin=0.8in,a4paper]{geometry}    
\usepackage{cite} 
\usepackage{caption}
\usepackage{float}
\usepackage{tikz}
\usepackage{authblk}
\usepackage[colorlinks=true,linkcolor=blue,citecolor=blue,urlcolor=blue]{hyperref}
\usepackage{parskip}
\usepackage{enumitem}
 \usepackage{booktabs}
\usepackage{csquotes}
 \usepackage{lmodern}
\usepackage{orcidlink}
\usepackage{xcolor}

\theoremstyle{plain}
\newtheorem{theorem}{Theorem}
\newtheorem{corollary}{Corollary}
\newtheorem{lemma}{Lemma}
\newtheorem{proposition}{Proposition}

\theoremstyle{definition}
\newtheorem{definition}{Definition}

\theoremstyle{remark}
\newtheorem{remark}{Remark}

\DeclareRobustCommand{\seqnum}[1]{%
  \ifmmode
    \text{\href{https://oeis.org/#1}{\textcolor{blue}{{\normalfont\ttfamily #1}}}}%
  \else
    \href{https://oeis.org/#1}{\textcolor{blue}{{\normalfont\ttfamily #1}}}%
  \fi
}

\author{El-Mehdi Mehiri \orcidlink{0000-0002-7164-3658}\\
Mines Saint-Etienne, CMP, Department of Manufacturing Sciences and Logistics, F-13120 Gardanne, France.\\
\url{elmehdi.mehiri@emse.fr}\\
\url{mehiri314@gmail.com}}

\title{\textbf{Bijections Between Smirnov Words\\and Hamiltonian Cycles \\in Complete Multipartite Graphs\thanks{\textcopyright   E.-M.Mehiri 2026}
}}

\date{19-06-2026}

\begin{document}

\maketitle

\begin{abstract}
\noindent We establish a bijective correspondence between Smirnov words with balanced letter multiplicities and Hamiltonian paths in complete $m$-partite graphs $K_{n,n,\ldots,n}$. This bijection allows us to derive closed inclusion-exclusion formulas for the number of Hamiltonian cycles in such graphs. We further extend the enumeration to the generalized nonuniform case $K_{n_1,n_2,\ldots,n_m}$.  We also provide an asymptotic analysis based on Stirling's approximation, which yields compact factorial expressions and logarithmic expansions describing the growth of the number of Hamiltonian cycles in the considered graphs.   

\end{abstract}



\section{Introduction}

Smirnov words are words over a finite alphabet in which no two consecutive letters are identical \cite{flajolet2009analytic}.  For   recent applications of Smirnov words, the reader is referred to \cite{freiberg2017application,SHARESHIAN2016497}.

In this paper, we establish a new bijection between Smirnov words  with equal letter multiplicities and Hamiltonian paths in complete $m$-partite graphs $K_{n,n,\ldots,n}$. 
This correspondence allows us to count Hamiltonian cycles in $K_{n,n,\ldots,n}$
by purely combinatorial arguments.

Determining whether a given graph admits a Hamiltonian cycle is a well-known NP-complete problem \cite{GareyJohnson1979}. An even more challenging task is to determine the number of distinct Hamiltonian cycles a graph possesses. Owing to this computational difficulty, numerous studies have focused on developing enumeration techniques for specific graph classes, or on establishing upper and lower bounds for the number of Hamiltonian cycles; see, e.g.,  \cite{bodroza2013enumeration,Jacobsen_2007,Thomassen_1996,SARKOZY2003237,SCHWENK198953,Haythorpe02102018,FRIEDGUT_KAHN_2005}.

In the special cases most directly related to complete multipartite graphs, several classical
enumeration results are known. The number of undirected Hamiltonian cycles in the complete
graph $K_m$ is $\frac{(m-1)!}{2}$,  whereas the complete bipartite graph $K_{n,n}$ has $\frac{(n!)^2}{2n}$  undirected Hamiltonian cycles. More generally, Singmaster~\cite{SINGMASTER19751}
enumerated Hamiltonian cycles with a prescribed initial vertex in the cocktail party graph
$K_{2,2,\ldots,2}$. Later, Hor\'ak and T\r{u}v\'arek~\cite{horak1979hamiltonian}
obtained a recursive formula for the number of Hamiltonian cycles in the complete multipartite
graph $K_{n_1,n_2,\ldots,n_m}$. In the balanced case $K_{n,n,\ldots,n}$,
Krasko \emph{et al.}~\cite{krasko2017enumeration} derived an enumeration formula by using
a bijection between Hamiltonian paths in octahedra and loopless generalized chord diagrams,
extending a bijective construction from their earlier work.

The purpose of the present paper is to give a different and direct word-theoretic approach to
this enumeration problem. We establish a bijection between Hamiltonian paths in the complete
multipartite graph $K_{n,n,\ldots,n}$ whose endpoints belong to distinct partite sets and
Smirnov words of length $mn$ with balanced multiplicities and distinct initial and terminal
letters. This correspondence transforms the Hamiltonian cycle counting problem into the
enumeration of constrained Smirnov words. It yields exact recurrence and inclusion-exclusion
formulas, and hence a closed expression for the number of undirected Hamiltonian cycles in
$K_{n,n,\ldots,n}$. We also derive asymptotic and logarithmic estimates using Stirling's
approximation, making explicit the factorial growth of the resulting sequence. Finally, the same
framework is extended to the nonuniform multipartite graph $K_{n_1,n_2,\ldots,n_m}$, as well
as to directed and circular variants of the corresponding Smirnov word model. In particular,
this provides an answer to a problem posed by Donald E. Knuth (2025) in \emph{The American Mathematical Monthly}~\cite[Problem 12561]{Ullman21102025}.

The paper is organized as follows. Section~\ref{sec:AdmissibleWords} introduces the multipartite alphabet, the class of
admissible words, and the reduction from labeled words to balanced color-words. Section~\ref{sec:Enumeration} is
devoted to the enumeration of these color-words under adjacency and endpoint constraints,
including both a recurrence and an inclusion-exclusion formula. In Section~\ref{sec:FromTo}, we construct
the bijection between admissible words and Hamiltonian paths in $K_{n,n,\ldots,n}$, and use it
to obtain the corresponding formula for undirected Hamiltonian cycles. Section~\ref{sec:Asymptotic} gives the
asymptotic analysis of the obtained formulas. Section~\ref{sec:extensions} extends the method to nonuniform complete multipartite graphs and discusses the directed-cycle variant. Finally,
Section~\ref{sec:Conclusion} concludes the paper and suggests further directions.

\section{Admissible Words}\label{sec:AdmissibleWords}

Let the alphabet $\mathcal{A}=\bigcup\limits_{i=1}^{m}\mathcal{A}_{i}$, where $|\mathcal{A}_{i}|=n$, and $\mathcal{A}_{i}\cap \mathcal{A}_{j}=\varnothing$, for all $i,j\in [m]= \{1,\ldots,m\}$, with $i\neq j$.

\begin{definition}
A word $w=w_1w_2\cdots w_{mn}$ over the alphabet $\mathcal{A}$ is called \emph{admissible} if:
\begin{enumerate}
    \item each letter of $\mathcal{A}$ appears exactly once in $w$;
    \item consecutive letters come from distinct subsets (Smirnov word), i.e.
    $w_t \in \mathcal{A}_i,\ w_{t+1} \in \mathcal{A}_j \implies i\neq j;$
    \item the first and last letters belong to different subsets, i.e.
    $w_1 \in \mathcal{A}_i,\ w_{mn} \in \mathcal{A}_j \implies i\neq j.$
\end{enumerate}
The set of all admissible words is denoted $\mathcal{W}_m(n)$,
and its cardinality $W_m(n)=|\mathcal{W}_m(n)|$.
\end{definition}

\begin{definition}
Associate to each subset $\mathcal{A}_i$ a color $i\in[m]$.
Replacing each letter of $w$ by its color yields a
\emph{color-word} over the alphabet $[m]$.
\end{definition}

Let $S_m(n)$ denote the number of color-words of length $mn$
that use each color exactly $n$ times, have no consecutive equal colors,
and begin and end with distinct colors.

\begin{proposition}\label{prop:reduction}
For all integers $m,n\geq1$, we have 
\begin{equation*} 
W_m(n) = S_m(n) (n!)^m.
\end{equation*}
\end{proposition}
\begin{proof}
Every color-word pattern corresponds to $(n!)^m$ admissible words, since the $n$ letters of each subset $\mathcal{A}_i$ may occupy their positions in any order. Hence, the result.
\end{proof}

\section{Enumeration of Color-Words}\label{sec:Enumeration}

\begin{definition}
Let $F_m(a_1,\dots,a_m; s\!\to\!r)$ be the number of color-words, that:
\begin{enumerate}
    \item begin with the color $s$ and end with the color $r$,
    \item have no two consecutive equal colors,
    \item use exactly $a_i$ additional  occurrences of color $i$ after the first letter (so the total length is $1+\sum_i a_i$).
\end{enumerate}
\end{definition}

\begin{lemma}\label{lem:rec}
Let $m\geq 2$, for all nonnegative $a_i$ and $s,r\in[m]$, $F_m(a_1,\dots,a_m; s\!\to\!r)$ satisfy the   recurrence 
\begin{align}
F_m(0,\dots,0; s\!\to\!r)
&=\mathbf{1}_{\{s=r\}}\label{eq:rec-base}\\[0.4em]
F_m(a_1,\dots,a_m; s\!\to\!r)
&=
\sum_{\substack{x=1\\x\neq s}}^{m}
F_m(a_1-\mathbf{1}_{x=1}, \dots, a_m-\mathbf{1}_{x=m}; x\!\to\!r),
\label{eq:recFm}
\end{align}
with the convention that $F_m(\ldots)=0$ whenever some argument $a_i-\mathbf{1}_{x=i}$ is negative, and where the indicator function 
\begin{equation*}
    \mathbf{1}_{x=y}=\begin{cases}
        1,&\text{if $x=y$;}\\
        0,&\text{otherwise.}
    \end{cases}
\end{equation*}
\end{lemma}

\begin{proof}
We proceed by induction  in two steps: base case and inductive (decomposition) step.

For $(a_1,\dots,a_m)=(0,\dots,0)$,   the total length is $1$, so the only possible word is the length $1$ word consisting of the single color $s$. This word ends in $r$ if and only if $r=s$. Hence $F_m(0,\dots,0; s\!\to\!r)=\mathbf{1}_{\{s=r\}}$, proving \eqref{eq:rec-base}.

Assume now $\sum_{i=1}^m a_i=mn\ge 1$, i.e., there are still $mn$ letters to place after the first position.
Every word counted by $F_m(a_1,\dots,a_m; s\!\to\!r)$ has the form
$$
\underbrace{s}_{\text{fixed first}}\;  \underbrace{x}_{\text{second}}\;  \underbrace{(\text{a valid tail of length }mn-1)}_{\text{no equal adjacencies, ends at }r},
$$
where the \emph{second} color $x$ must satisfy (i) $x\neq s$ (to avoid an immediate adjacency violation) and (ii) $a_x\ge 1$ (since one occurrence of color $x$ is used at the second position).

Fix such an $x\in[m]\setminus\{s\}$ with $a_x\ge 1$. Removing the first two letters $(s,x)$ from the word leaves a tail of length $mn-1$ that:
\begin{itemize}
    \item begins with color $x$ (because the third position must differ from $x$ and belongs to the tail's first step from the perspective of the reduced problem),
    \item ends with color $r$ (the overall last letter is unchanged),
    \item uses exactly $a_i' := a_i - \mathbf{1}_{x=i}$ occurrences of color $i$ for each $i$,
    \item has no equal adjacencies (this is inherited from the original word).
\end{itemize}
By definition, the number of such tails is precisely $
F_m\bigl(a_1-\mathbf{1}_{x=1}, \dots, a_m-\mathbf{1}_{x=m}; x\!\to\!r\bigr)
$.

Conversely, given any tail counted by $F_m(a_1-\mathbf{1}_{x=1},\dots,a_m-\mathbf{1}_{x=m}; x\!\to\!r)$, prefixing it by the two-letter block $(s,x)$ produces a valid word counted by $F_m(a_1,\dots,a_m; s\!\to\!r)$, because:
\begin{itemize}
    \item $x\neq s$ enforces the no-adjacency condition at the boundary between the first and second letters,
    \item the tail, by hypothesis, has no equal adjacencies internally,
    \item the total color counts match exactly: we used $1$ occurrence of $x$ (hence reduced $a_x$ by $1$) and left all $a_i$ for $i\neq x$ unchanged.
\end{itemize}

Therefore, for each fixed $x\neq s$ with $a_x\ge 1$, there is a bijection  between words counted by $F_m(a_1,\dots,a_m; s\!\to\!r)$ whose second color is $x$,
  and  
tails counted by $F_m(a_1-\mathbf{1}_{x=1},\dots,a_m-\mathbf{1}_{x=m}; x\!\to\!r)$.
 
The sets corresponding to different choices of $x$ are disjoint (the second letter is uniquely determined), and their union is exactly the set of all words counted by $F_m(a_1,\dots,a_m; s\!\to\!r)$ (every valid word has some second color $x\neq s$ with $a_x\ge 1$). Summing the cardinalities of these disjoint classes over all $x\in[m]\setminus\{s\}$ yields \eqref{eq:recFm},  under the convention that any term with a negative argument (i.e., $a_x-1<0$) contributes $0$, as there can be no word whose second letter is $x$ if no $x$'s remain to place.
\end{proof}

Let $F_m^{(s\to r)}(n)$ denote the number of $mn$-length words over the alphabet $[m]$, that  use each color exactly $n$ times,  begin with color $s$ and end with color $r$, and  have no two consecutive equal colors.

\begin{lemma}\label{lem:symmetry}
Fix $m\ge 2$ and $n\ge 1$. For all $s,r\in  [m]$,
if $s\neq r$ and $s',r'\in[m]$ with $s'\neq r'$, then
\[
F_m^{(s\to r)}(n)=F_m^{(s'\to r')}(n).
\]
In particular, since there are $m(m-1)$ ordered endpoint pairs $(s,r)$ with $s\neq r$,
\[
S_m(n)=\sum_{\substack{s,r\in[m]\\ s\neq r}} F_m^{(s\to r)}(n)
= m(m-1) F_m^{(1\to 2)}(n).
\]
\end{lemma}

\begin{proof}
For fixed ordered pairs $(s,r)$ and $(s',r')$ with $s\neq r$ and
$s'\neq r'$, choose a permutation $\pi\in S_m$ such that $\pi(s)=s'$ and $\pi(r)=r'$.  Such a permutation exists because the symmetric group acts transitively on
ordered pairs of distinct elements of $[m]$.

We now relabel each color according to $\pi$. More precisely, if $w=w_1w_2\cdots w_{mn}$ is a color-word over $[m]$, define
\[
    \Phi_\pi(w):=\pi(w_1)\pi(w_2)\cdots \pi(w_{mn}).
\]
Suppose that $w$ is counted by $F_m^{(s\to r)}(n)$. Then $w$ begins with
$s$ and ends with $r$, and therefore $\Phi_\pi(w)$ begins with
$\pi(s)=s'$ and ends with $\pi(r)=r'$. Moreover, since $\pi$ is a bijection
of $[m]$, it preserves multiplicities: each color still appears exactly
$n$ times after the relabeling. Finally, the Smirnov condition is also
preserved. Indeed, whenever $w_t\neq w_{t+1}$, the injectivity of $\pi$
implies $\pi(w_t)\neq \pi(w_{t+1})$.
 
Thus $\Phi_\pi(w)$ has no two consecutive equal colors. Hence
$\Phi_\pi$ maps the family of words counted by $F_m^{(s\to r)}(n)$ into
the family of words counted by $F_m^{(s'\to r')}(n)$.

This map is invertible. Its inverse is obtained by relabeling colors
according to $\pi^{-1}$. In particular, $\Phi_{\pi^{-1}}\circ \Phi_\pi=\mathrm{id}$ and $\Phi_\pi\circ \Phi_{\pi^{-1}}=\mathrm{id}$. Therefore $\Phi_\pi$ is a bijection between the two families of words, and
so $F_m^{(s\to r)}(n)=F_m^{(s'\to r')}(n)$.

It remains only to sum over the possible endpoint pairs. By definition,
\[
    S_m(n)
    =
    \sum_{\substack{s,r\in[m]\\ s\neq r}}
    F_m^{(s\to r)}(n).
\]
By the bijection just proved, all terms in this sum are equal. Since there
are $m(m-1)$ ordered pairs $(s,r)$ with $s\neq r$, each term is equal to
$F_m^{(1\to 2)}(n)$, and hence
\[
    S_m(n)=m(m-1)F_m^{(1\to 2)}(n).
\]
This proves the result.
\end{proof}
 
Theorem~\ref{th:closed} gives a closed inclusion-exclusion formula for $F_m(a_1,\dots,a_m; s\!\to\!r)$. 

\begin{theorem}\label{th:closed}
Let $m\ge2$, and let $a_1,\dots,a_m\ge0$ with total length $mn=\sum_{i=1}^m a_i$.
For fixed endpoints $s,r\in[m]$, define the adjusted counts $a_i' := a_i - \mathbf{1}_{\{i=s\}} - \mathbf{1}_{\{i=r\}}$. 
 
If any $a_i'<0$ (that is, if color $i$ would be used fewer than twice, once at each endpoint),
then $F_m(a_1,\dots,a_m; s\!\to\!r)=0$.
Otherwise:
\begin{equation}\label{eq:closedFm}
F_m(a_1,\dots,a_m; s\!\to\!r)
=
\!\!\sum_{k_1=0}^{a_1-1}\cdots\sum_{k_m=0}^{a_m-1}
(-1)^{k_1+\cdots+k_m}
\!\left(\prod_{i=1}^m \binom{a_i-1}{k_i}\right)
\frac{(mn-2-\sum_i k_i)!}{\prod_{i=1}^m (a_i'-k_i)!}.
\end{equation}
\end{theorem}

\begin{proof}
 
If \(a_i'<0\) for some \(i\), then the prescribed endpoints require more
occurrences of color \(i\) than are available, and therefore no such word exists.
In this case both sides are interpreted as zero. We may thus assume that
\(a_i'\geq 0\) for all \(i\).

We first fix the endpoints. The first letter is required to be \(s\), and the
last letter is required to be \(r\). After these two letters have been fixed,
there remain \(mn-2\) positions to be filled with \(a_i'\) copies of color \(i\),
for each \(i\in[m]\). If the adjacency condition were ignored, the number of
such words would be $\frac{(mn-2)!}{\prod_{i=1}^m a_i'!}$.
 
It remains to impose the Smirnov condition, namely the condition that no two
consecutive letters are equal. We do this by inclusion-exclusion.  For a fixed color \(i\), there are \(a_i-1\) possible adjacencies
between the \(a_i\) occurrences of that color. Suppose that \(k_i\) of these
adjacencies are chosen to be merged. This can be done in $\binom{a_i-1}{k_i}$ ways, and after the merging the \(a_i\) occurrences of color \(i\) are replaced
by \(a_i-k_i\) blocks of color \(i\). If we perform this construction
simultaneously for all colors, with $\mathbf{k}=(k_1,\ldots,k_m)$, $0\leq k_i\leq a_i-1$,  then the total number of blocks becomes $mn-\sum_{i=1}^m k_i$. The sign associated with this choice in the inclusion-exclusion sum is $(-1)^{k_1+\cdots+k_m}$.

Once the merges have been chosen, the first and last blocks are still prescribed
to have colors \(s\) and \(r\), respectively. Hence the number of remaining
interior blocks is $mn-2-\sum_{i=1}^m k_i$. 
 
For each color \(i\), the number of interior blocks of color \(i\) is $a_i'-k_i$,  
because one endpoint block is removed if \(i=s\), one endpoint block is removed
if \(i=r\), and \(k_i\) merges have been performed among the occurrences of
color \(i\). Therefore, for this fixed vector \(\mathbf{k}\), the number of
corresponding block arrangements is
\[
    \frac{\left(mn-2-\sum_{i=1}^m k_i\right)!}
         {\prod_{i=1}^m (a_i'-k_i)!}.
\]
As usual, a term is interpreted as zero whenever one of the factorials in the
denominator has a negative argument.

Multiplying by the number of ways to choose the merges for each color and then
summing over all possible vectors \(\mathbf{k}\), inclusion-exclusion gives \eqref{eq:closedFm}.  
\end{proof}

\section{From Words to Hamiltonian Cycles}\label{sec:FromTo}

Let $K_{n,n,\dots,n}$ be the complete $m$-partite graph
with vertex partition $V = V_1 \cup \cdots \cup V_m$,
$|V_i|=n$, and edges only between different parts.

\begin{theorem}\label{th:bij}
There exists a bijection
between admissible words \(w\in\mathcal{W}_m(n)\) and Hamiltonian directed paths
in \(K_{n,n,\ldots,n}\) whose endpoints lie in distinct partite sets.
\end{theorem}

\begin{proof}
Fix, for each \(i\in[m]\), a bijection between the alphabet
\(\mathcal{A}_i\) and the partite set \(V_i\). Thus each letter of
\(\mathcal{A}_i\) may be identified with a unique vertex of \(V_i\).

Let \(w=w_1w_2\cdots w_{mn}\in\mathcal{W}_m(n)\). Replacing each letter
\(w_t\in\mathcal{A}_{i_t}\) by its corresponding vertex
\(v_t\in V_{i_t}\) gives a sequence $v_1\to v_2\to\cdots\to v_{mn}$. 
 
Since \(w\) is admissible, all letters appear exactly once, and hence all
vertices appear exactly once. Moreover, consecutive letters of \(w\) belong
to distinct subsets of the alphabet, so consecutive vertices belong to distinct
partite sets. They are therefore adjacent in \(K_{n,n,\ldots,n}\). Hence the
sequence is a Hamiltonian directed path. The endpoint condition in the definition
of \(\mathcal{W}_m(n)\) also implies that the first and last vertices lie in
distinct partite sets.

Conversely, let $P: v_1\to v_2\to\cdots\to v_{mn}$ be a Hamiltonian directed path in \(K_{n,n,\ldots,n}\) whose endpoints lie in
distinct partite sets. Replacing each vertex \(v_t\in V_i\) by its corresponding
letter in \(\mathcal{A}_i\) gives a word $w_P=w_1w_2\cdots w_{mn}$. 
 
Since \(P\) is Hamiltonian, every letter appears exactly once. Since edges in a
complete multipartite graph join only vertices from distinct partite sets, no two
consecutive letters of \(w_P\) belong to the same subset \(\mathcal{A}_i\).
Finally, the endpoints of \(P\) lie in distinct partite sets, so the first and
last letters of \(w_P\) also belong to distinct subsets. Thus
\(w_P\in\mathcal{W}_m(n)\).

The two constructions are inverse to each other, and therefore establish the  bijection.
\end{proof}

\begin{corollary}\label{cor:Hm}
Let \(H_m(n)\) be the number of undirected Hamiltonian cycles in
\(K_{n,n,\ldots,n}\). Then
\begin{equation}\label{eq:Hm}
    H_m(n)
    =
    \frac{W_m(n)}{2mn}
    =
    \frac{(n!)^m}{2mn}S_m(n).
\end{equation}
\end{corollary}

\begin{proof}
By Theorem~\ref{th:bij}, \(W_m(n)\) counts the Hamiltonian directed paths in
\(K_{n,n,\ldots,n}\) whose endpoints lie in distinct partite sets. Since such
endpoints are adjacent, each of these paths can be closed into a Hamiltonian
cycle.

Conversely, every undirected Hamiltonian cycle can be read in two orientations
and can be opened at any of its \(mn\) vertices. Hence each undirected
Hamiltonian cycle gives exactly \(2mn\) Hamiltonian directed paths of the type
counted by \(W_m(n)\). Therefore, $H_m(n)=\frac{W_m(n)}{2mn}$. 
 
Using the reduction $ W_m(n)=(n!)^m S_m(n)$,  we obtain $H_m(n)=\frac{(n!)^m}{2mn}S_m(n)$.
\end{proof}

Tables~\ref{tab:first_values_S} and~\ref{tab:first_values_H} report the first computed values of 
$S_{m}(n)$ and $H_{m}(n)$, respectively, for $1\leq m\leq5$ and $1\leq n\leq7$. 
Whenever possible, the corresponding sequence identifier from the \textsc{OEIS} (On-Line Encyclopedia of Integer Sequences \cite{OEIS}) is also provided.

\begin{table}[ht]
\resizebox{\columnwidth}{!}{
\begin{tabular}{lll}
\toprule
 $m$& $S_{m}(n)$  & OEIS                              \\ \midrule
1  & 0, 0, 0, 0, 0, 0, 0,\ldots  & -  \\  
2  & 2, 2, 2, 2, 2, 2, 2,\ldots & - \\  
3 & 6, 24, 132, 804, 5196, 34872, 240288,\ldots & \seqnum{A110707}$(n)$ \\  
4 & 24, 744, 33960, 1820232, 106721784, 6627719256, 428434032456,\ldots  & -                                 \\  
5 & 120, 35160, 16841160, 9960343920, 6633577962720, 4768569352231680,  3615532424230568640,\ldots   & -  \\ \bottomrule
\end{tabular}
 }
\caption{First values of $S_{m}(n)$.}
\label{tab:first_values_S}
\end{table}

\begin{table}[ht]
\centering
\resizebox{\columnwidth}{!}{
\begin{tabular}{lll}
\toprule
$m$ &  $H_{m}(n)$ & OEIS         \\ \midrule
1 &  0, 0, 0, 0, 0, 0, \ldots  & -            \\ 
2 & 0, 1, 6, 72, 1440, 43200, \ldots & \seqnum{A010796}$(n-1)$  \\  
3 & 1, 16, 1584, 463104, 299289600, 361552896000,  \ldots  & \seqnum{A307924}$(n)$   \\  
4 & 3, 744, 1833840, 18872165376,  553245728256000,    37106744352952320000,   \ldots  &\seqnum{A381326}$(n)$   \\  
5 & 12, 56256, 4365228672,  1982761838641152, 2390860800330129294300,  15377981531746493634969600000,    \ldots & -   \\ \bottomrule
\end{tabular}
}
\caption{First values of $H_{m}(n)$.}\label{tab:first_values_H}

\end{table}

\section{Asymptotic analysis}\label{sec:Asymptotic}

The exact formulas obtained above are suitable for computation, but they also
give useful information about the growth of the number of Hamiltonian cycles.
We record here the main asymptotic estimate suggested by the
inclusion-exclusion formula and by the standard probabilistic interpretation
of the Smirnov constraint.

The case \(m=2\) is exceptional: a balanced binary Smirnov word is necessarily
alternating, and hence \(S_2(n)=2\). Consequently, $H_2(n)=\frac{(n!)^2}{2n}$,  which is the classical formula for the complete bipartite graph \(K_{n,n}\).
Thus, in the discussion below, we focus on the case \(m\geq 3\).

\begin{proposition}\label{prop:asymptotic-estimate}
Fix \(m\geq 3\), and set $q_m=1-\frac{1}{m}$.  As \(n\to\infty\), the dominant asymptotic estimate is
\begin{equation}\label{eq:AsymptoticSm}
    S_m(n)
    \approx
    m(m-1)
    \frac{(mn)!}{(n!)^m}
    q_m^{mn-1}.
\end{equation}
Consequently,
\begin{equation}\label{eq:AsymptoticHm}
    H_m(n)
    \approx
    \frac{(m-1)(mn)!}{2n}
    q_m^{mn-1}.
\end{equation}
\end{proposition}

\begin{proof}
There are $\frac{(mn)!}{(n!)^m}$ balanced color-words of length \(mn\), that is, words in which each of the
\(m\) colors appears exactly \(n\) times. Among these words, the Smirnov
condition excludes equal adjacent colors.

For large \(n\), the colors are nearly uniformly distributed along the word.
Thus, the probability that two consecutive positions have the same color is
approximately \(1/m\), and the probability that a given adjacent pair is
admissible is approximately $q_m=1-\frac{1}{m}$.  Moreover, there are \(mn-1\) adjacent pairs, so the adjacency constraint contributes the
factor \(q_m^{mn-1}\).

Finally, the endpoint condition requires the first and last colors to be
distinct. Since there are \(m(m-1)\) ordered choices of distinct endpoint
colors, we obtain the estimate \eqref{eq:AsymptoticSm}. 
 
Using \eqref{eq:Hm} and \eqref{eq:AsymptoticSm}, we obtain  \eqref{eq:AsymptoticHm}. 
\end{proof}

\begin{corollary}[Logarithmic form]\label{cor:log-asymptotic}
For fixed \(m\geq 3\), the estimate of Proposition~\ref{prop:asymptotic-estimate} gives the following logarithmic form:
\begin{equation}\label{eq:AsymtoticLogaHm}
    \log H_m(n)
    =
    \log(m-1)-\log(2n)
    +\log((mn)!)
    +(mn-1)\log q_m
    +O(1).
\end{equation}
Equivalently,
\begin{equation}\label{eq:AsymtoticLogaHm22}
    \log H_m(n)
    =
    mn\log n
    +
    n\bigl(m\log m+m\log q_m-m\bigr)
    +O(\log n).
\end{equation}
In particular, $\log H_m(n)=mn\log n+\Theta(n)$, which shows that the number of Hamiltonian cycles grows factorially, with an
exponential correction coming from the Smirnov adjacency constraint.
\end{corollary}

\begin{proof}
Taking logarithms in the estimate \eqref{eq:AsymptoticHm} gives \eqref{eq:AsymtoticLogaHm}. 
 
Using Stirling's formula,
\[
    \log((mn)!)
    =
    mn\log(mn)-mn+O(\log n),
\]
and writing \(\log(mn)=\log n+\log m\), we obtain \eqref{eq:AsymtoticLogaHm22}. Thus, the final statement follows immediately.
\end{proof}

It is also useful to look at successive ratios. If
\[
    R_m(n)=\frac{H_m(n+1)}{H_m(n)},
\]
then Proposition~\ref{prop:asymptotic-estimate} gives
\[
    R_m(n)
    \approx
    \frac{(m(n+1))!}{(mn)!}
    \frac{n}{n+1}
    q_m^m.
\]
Equivalently,
\[
    R_m(n)
    \approx
    \left(\prod_{k=1}^{m}(mn+k)\right)
    \frac{n}{n+1}
    q_m^m.
\]
Thus the growth of \(H_m(n)\) is mainly driven by the factorial factor
\((mn)!\), while the term \(q_m^{mn-1}\) accounts for the loss caused by the
adjacency restriction.

\section{Extensions and generalizations}\label{sec:extensions}

The word-path correspondence developed above extends naturally to complete
multipartite graphs with unequal part sizes. In this section, we briefly record
this extension, together with a closely related variant.

\subsection{The nonuniform complete multipartite graph}
Let \(K_{n_1,\ldots,n_m}\) be the complete \(m\)-partite graph with partite
sets of sizes \(n_1,\ldots,n_m\).  Let $\mathcal{A}
    =
    \mathcal{A}_1\cup\cdots\cup \mathcal{A}_m$, $|\mathcal{A}_i|=n_i\geq 1$,   and let $ N=\sum_{i=1}^m n_i$. 
 
We denote by \(S(n_1,\ldots,n_m)\) the number of color-words of length \(N\)
over the alphabet \([m]\) in which color \(i\) appears exactly \(n_i\) times,
no two consecutive colors are equal, and the first and last colors are distinct.

For fixed endpoints \(s\neq r\), let
\(F(n_1,\ldots,n_m;s\!\to\! r)\) be the number of such words that begin with
\(s\) and end with \(r\). 

Set $n_i'
    =
    n_i-\mathbf{1}_{\{i=s\}}-\mathbf{1}_{\{i=r\}}$.  If \(n_i'<0\) for some \(i\), then $F(n_1,\ldots,n_m;s\!\to\! r)=0$.  Otherwise, the same inclusion-exclusion argument used in the balanced case
gives
\begin{equation}\label{eq:IE-nonuniform}
F(n_1,\ldots,n_m;s\!\to\! r)
=
\sum_{k_1=0}^{n_1-1}
\cdots
\sum_{k_m=0}^{n_m-1}
(-1)^{\sum_i k_i}
\left(
    \prod_{i=1}^m \binom{n_i-1}{k_i}
\right)
\frac{
    \left(N-2-\sum_i k_i\right)!
}{
    \prod_{i=1}^m (n_i'-k_i)!
}.
\end{equation}
As usual, a term is interpreted as zero whenever a factorial in the denominator
has a negative argument. Summing over the possible endpoint pairs yields
\begin{equation}
    S(n_1,\ldots,n_m)
    =
    \sum_{\substack{s,r\in[m]\\ s\neq r}}
    F(n_1,\ldots,n_m;s\!\to\! r).
\end{equation}
\begin{theorem}\label{th:K-nonuniform}
The number of undirected Hamiltonian
cycles in \(K_{n_1,\ldots,n_m}\) is
\begin{equation}\label{eq:H-nonuniform}
    H(n_1,\ldots,n_m)
    =
    \frac{\prod_{i=1}^m n_i!}{2N}
    S(n_1,\ldots,n_m).
\end{equation}
\end{theorem}

\begin{proof}
The bijection from Theorem~\ref{th:bij} applies without change. A color-word
with multiplicities \(n_1,\ldots,n_m\) describes the sequence of partite sets
visited by a Hamiltonian path, and the Smirnov condition ensures that
consecutive vertices lie in distinct partite sets.

Once a color pattern is fixed, the \(n_i\) vertices in the \(i\)-th part can be
assigned to the \(n_i\) occurrences of color \(i\) in \(n_i!\) ways. Hence the
number of corresponding labeled Hamiltonian directed paths with endpoints in
distinct parts is $\prod_{i=1}^m n_i!\; S(n_1,\ldots,n_m)$.
 
Each undirected Hamiltonian cycle gives exactly \(2N\) such directed paths,
corresponding to the two orientations and the \(N\) possible choices of the
starting vertex. Dividing by \(2N\) gives \eqref{eq:H-nonuniform}.
\end{proof}

\begin{remark}
For \(m=3\), Theorem~\ref{th:K-nonuniform} gives an explicit answer for the
complete tripartite graph \(K_{p,q,r}\). Namely,
\[
    H(p,q,r)
    =
    \frac{p!q!r!}{2(p+q+r)}
    S(p,q,r),
\]
where \(S(p,q,r)\) is computed from the inclusion-exclusion formula
\eqref{eq:IE-nonuniform}. This answers the enumeration problem for undirected
Hamiltonian cycles in \(K_{p,q,r}\) posed by Donald E.~Knuth in
\cite{Ullman21102025}.
\end{remark}

\subsection{Directed cycles}

If the reverse orientation of a Hamiltonian cycle is counted as a different
cycle, then only the \(N\) cyclic rotations have to be identified. Therefore,
in the directed setting one obtains the following immediate variant.

\begin{corollary}\label{cor:directed}
The number of directed Hamiltonian cycles in \(K_{n_1,\ldots,n_m}\) is
\[
    H^{\rightarrow}(n_1,\ldots,n_m)
    =
    \frac{\prod_{i=1}^m n_i!}{N}
    S(n_1,\ldots,n_m).
\]
Equivalently,
\[
    H^{\rightarrow}(n_1,\ldots,n_m)
    =
    2H(n_1,\ldots,n_m).
\]
\end{corollary}

\begin{proof}
The proof is the same as that of Theorem~\ref{th:K-nonuniform}, except that a
directed Hamiltonian cycle has \(N\) linear representations rather than \(2N\),
since the two orientations are now distinct.
\end{proof}

\section{Conclusion}\label{sec:Conclusion}

In this paper, we established a direct bijection between admissible Smirnov words and
Hamiltonian paths in complete multipartite graphs. This correspondence allowed us to
derive exact enumeration formulas for Hamiltonian cycles in the balanced graph
\(K_{n,n,\ldots,n}\), using recurrence relations and inclusion-exclusion arguments for
constrained color-words. We also studied the asymptotic growth of these numbers and
extended the method to the nonuniform graph \(K_{n_1,\ldots,n_m}\).

Several natural directions remain open. One may seek sharper asymptotic estimates,
especially in the nonuniform case where the part sizes grow with fixed proportions. It
would also be interesting to study analogous bijections for other families of multipartite
or highly symmetric graphs, and to investigate whether related word models can be used
to count Hamiltonian cycles under additional constraints, such as prescribed endpoints,
forbidden transitions, or directed edge restrictions.







\bibliographystyle{plainurl} 
\bibliography{bibou}

@misc{OEIS,
  author       = {{OEIS Foundation Inc.}},
  title        = {The On-Line Encyclopedia of Integer Sequences},
  year         = {2025},
  note         = {Published electronically},
  url          = {https://oeis.org}
}

@book{GareyJohnson1979,
  author       = {Garey, Michael R. and Johnson, David S.},
  title        = {Computers and Intractability: A Guide to the Theory of {NP}-Completeness},
  publisher    = {W. H. Freeman and Company},
  address      = {San Francisco},
  year         = {1979}
}

@book{flajolet2009analytic,
  author       = {Flajolet, Philippe and Sedgewick, Robert},
  title        = {Analytic Combinatorics},
  publisher    = {Cambridge University Press},
  address      = {Cambridge},
  year         = {2009},
  doi          = {10.1017/CBO9780511801655}
}

@article{freiberg2017application,
  author       = {Freiberg, Uta and Heuberger, Clemens and Prodinger, Helmut},
  title        = {Application of {Smirnov} Words to Waiting Time Distributions of Runs},
  journal      = {The Electronic Journal of Combinatorics},
  volume       = {24},
  number       = {3},
  pages        = {P3.55},
  year         = {2017},
  doi          = {10.37236/5753}
}

@article{SHARESHIAN2016497,
  author       = {Shareshian, John and Wachs, Michelle L.},
  title        = {Chromatic Quasisymmetric Functions},
  journal      = {Advances in Mathematics},
  volume       = {295},
  pages        = {497--551},
  year         = {2016},
  issn         = {0001-8708},
  doi          = {10.1016/j.aim.2015.12.018}
}

@article{bodroza2013enumeration,
  author       = {Bodro{\v{z}}a-Panti{\'c}, Olga and Panti{\'c}, Bojana and Panti{\'c}, Ilija and Bodro{\v{z}}a-Solarov, Marija},
  title        = {Enumeration of {Hamiltonian} Cycles in Some Grid Graphs},
  journal      = {MATCH Communications in Mathematical and in Computer Chemistry},
  volume       = {70},
  number       = {1},
  pages        = {181--204},
  year         = {2013}
}

@article{Jacobsen_2007,
  author       = {Jacobsen, Jesper Lykke},
  title        = {Exact Enumeration of {Hamiltonian} Circuits, Walks and Chains in Two and Three Dimensions},
  journal      = {Journal of Physics A: Mathematical and Theoretical},
  volume       = {40},
  number       = {49},
  pages        = {14667--14678},
  year         = {2007},
  doi          = {10.1088/1751-8113/40/49/003}
}

@article{Thomassen_1996,
  author       = {Thomassen, Carsten},
  title        = {On the Number of {Hamiltonian} Cycles in Bipartite Graphs},
  journal      = {Combinatorics, Probability \& Computing},
  volume       = {5},
  number       = {4},
  pages        = {437--442},
  year         = {1996},
  doi          = {10.1017/S0963548300002182}
}

@article{SARKOZY2003237,
  author       = {S{\'a}rk{\"o}zy, G{\'a}bor N. and Selkow, Stanley M. and Szemer{\'e}di, Endre},
  title        = {On the Number of {Hamiltonian} Cycles in {Dirac} Graphs},
  journal      = {Discrete Mathematics},
  volume       = {265},
  number       = {1--3},
  pages        = {237--250},
  year         = {2003},
  issn         = {0012-365X},
  doi          = {10.1016/S0012-365X(02)00582-4}
}

@article{SCHWENK198953,
  author       = {Schwenk, Allen J.},
  title        = {Enumeration of {Hamiltonian} Cycles in Certain Generalized {Petersen} Graphs},
  journal      = {Journal of Combinatorial Theory, Series B},
  volume       = {47},
  number       = {1},
  pages        = {53--59},
  year         = {1989},
  issn         = {0095-8956},
  doi          = {10.1016/0095-8956(89)90064-6}
}

@article{Haythorpe02102018,
  author       = {Haythorpe, Michael},
  title        = {On the Minimum Number of {Hamiltonian} Cycles in Regular Graphs},
  journal      = {Experimental Mathematics},
  volume       = {27},
  number       = {4},
  pages        = {426--430},
  year         = {2018},
  publisher    = {Taylor \& Francis},
  doi          = {10.1080/10586458.2017.1306813}
}

@article{FRIEDGUT_KAHN_2005,
  author       = {Friedgut, Ehud and Kahn, Jeff},
  title        = {On the Number of {Hamiltonian} Cycles in a Tournament},
  journal      = {Combinatorics, Probability \& Computing},
  volume       = {14},
  number       = {5--6},
  pages        = {769--781},
  year         = {2005},
  doi          = {10.1017/S0963548305006863}
}

@article{SINGMASTER19751,
  author       = {Singmaster, David},
  title        = {{Hamiltonian} Circuits on the {$n$}-Dimensional Octahedron},
  journal      = {Journal of Combinatorial Theory, Series B},
  volume       = {19},
  number       = {1},
  pages        = {1--4},
  year         = {1975},
  issn         = {0095-8956},
  doi          = {10.1016/0095-8956(75)90069-6}
}

@article{horak1979hamiltonian,
  author       = {Hor{\'a}k, Peter and Tov{\'a}rek, Leo{\v{s}}},
  title        = {On {Hamiltonian} Cycles of Complete {$n$}-Partite Graphs},
  journal      = {Mathematica Slovaca},
  volume       = {29},
  number       = {1},
  pages        = {43--47},
  year         = {1979},
  publisher    = {Mathematical Institute of the Slovak Academy of Sciences},
  url          = {https://dml.cz/dmlcz/132268}
}

@article{krasko2017enumeration,
  author       = {Krasko, Evgeniy and Omelchenko, Alexander},
  title        = {Enumeration of Chord Diagrams without Loops and Parallel Chords},
  journal      = {The Electronic Journal of Combinatorics},
  volume       = {24},
  number       = {3},
  pages        = {P3.43},
  year         = {2017},
  doi          = {10.37236/6037}
}

@article{Ullman21102025,
  author       = {Ullman, Daniel H. and Velleman, Daniel J. and Wagon, Stan and West, Douglas B.},
  title        = {Problems and Solutions},
  journal      = {The American Mathematical Monthly},
  volume       = {132},
  number       = {9},
  pages        = {930--940},
  year         = {2025},
  publisher    = {Taylor \& Francis},
  doi          = {10.1080/00029890.2025.2542705}
}

\end{document}